\documentclass[a4paper, reqno, 11pt]{amsart}

\usepackage[sc, noBBpl]{mathpazo}	% Palladio
\usepackage[
backend=biber,
style=alphabetic,
]{biblatex}
\usepackage{amssymb, amsmath, amsthm}

\usepackage{amsmath, amsthm, amssymb} 
\usepackage[all]{xy}
\usepackage[usenames,dvipsnames]{color}
\usepackage[T1]{fontenc}
\usepackage[utf8]{inputenc}
\usepackage{graphicx}
\usepackage{comment}
\usepackage[linkbordercolor={0.8 0.9 0.9}, citebordercolor={0.85 0.85 0.85}]{hyperref}
\usepackage{cleveref}    % cleveref has to be loaded after hyperref and amsmath
\usepackage[normalem]{ulem}

\numberwithin{equation}{section}

\theoremstyle{plain}
\newtheorem{theorem}{Theorem}[section]

\newtheorem{lemma}[theorem]{Lemma}

\theoremstyle{definition}
\newtheorem{definition}[theorem]{Definition}
\newtheorem{example}[theorem]{Example}
\theoremstyle{remark}
\newtheorem{remark}[theorem]{Remark}

  {\par\begin{tabular}{rcl}}%
  {\end{tabular}\par}

\makeatletter
\newcommand*\slot{\mathpalette\slot@{.5}}
\newcommand*\slot@[2]{\mathbin{\vcenter{\hbox{\scalebox{#2}{$\m@th#1\bullet$}}}}}
\makeatother

\newcommand{\lie}{\mathfrak}

\newcommand{\CC}{\mathbb{C}}
\newcommand{\DD}{\mathbb{D}}

\newcommand{\NN}{\mathbb{N}}
\newcommand{\RR}{\mathbb{R}}

\newcommand{\ZZ}{\mathbb{Z}}

\newcommand{\cA}{\mathcal{A}}
\newcommand{\cB}{\mathcal{B}}

\newcommand{\cK}{\mathcal{K}}

\newcommand{\cQ}{\mathcal{Q}}

\newcommand{\cU}{\mathcal{U}}
\newcommand{\cX}{\mathcal{X}}

\newcommand{\SU}{\mathrm{SU}}

\newcommand{\SO}{\mathrm{SO}}

\newcommand{\CCGq}{\CC[G_q]}
\newcommand{\CCXq}{\CC[\cX_q]}
\newcommand{\CCXp}{\CC[\cX_p]}
\newcommand{\CCDq}{\CC[\overline\DD_q]}
\newcommand{\CCDqTn}{\CC[\overline\DD_q/T_n]}

\newcommand{\leaf}{F}

\DeclareMathOperator{\Vect}{span}
\DeclareMathOperator{\CommSp}{CommSp}

\begin{document}

\title{Algebraic isomorphisms of quantized homogeneous spaces}
\author{Robert Yuncken}
\thanks{This research was supported by ANR project OpART (ANR-23-CE40-0016), and by COST Action CaLISTA CA21109 (European Cooperation in Science and Technology),  \href{https://www.cost.eu}{www.cost.eu}. 
The author is grateful for the support and hospitality of the Sydney Mathematics Research Institute and the University of Wollongong, where this article was written.}

\address{  Université de Lorraine, CNRS, IECL, F-57000 Metz, France}
\email{robert.yuncken@univ-lorraine.fr}


\date{\today}
\maketitle

\begin{abstract}
We describe a proof of the following folklore theorem:  If $\cX = G/K$ is the homogeneous space of a simply connected compact semisimple Lie group with Poisson-Lie stabilizers, then the $q$-deformed algebras of regular functions $\CC[\cX_q]$ with $0<q\leq1$ are mutually non-isomorphic as $*$-algebras.
\end{abstract}

\section{Introduction}

Fix $G$ a compact simply connected semisimple Lie group and $K$ a proper closed Poisson-Lie subgroup.  The homogeneous space $\cX = G/K$ admits a family of $q$-deformations, $0<q<\infty$, realized via their $*$-algebras of regular functions $\CCXq$.

The following theorem is folklore.

\begin{theorem}
  \label{theorem}
  $\CCXp \cong \CCXq$ if and only if $p=q^{\pm1}$.
\end{theorem}

In this expository note, we will give a proof of this fact.  The ideas are mostly not original.  A proof for the quantum spheres $\cX_q = S^{2n+1}_q$ was given by D'Andrea \cite{DAndrea:quantum_spheres}.  Krähmer \cite{Krahmer:Podles_spheres} proved a similar result for the nonstandard Podle\'{s} spheres.  

In the present context, the algebras $\CCGq$ are due to Soibelman \cite{Soibelman}, see also \cite{Woronowicz:pseudogroups}.  The algebras $\CCXq$ were introduced over the following years, but a systematic study of their structure was undertaken by Stokman and Dijkhuizen \cite{StoDij,Stokman:quantum_orbit_method}.  This was further developed by Neshveyev and Tuset  in \cite{NesTus:functions}. The essential points necessary for proving Theorem \ref{theorem} are already contained in the papers just listed.  

Giselsson \cite{Giselsson} proved that the enveloping $C^*$-algebras $C(\cX_q)$ are all isomorphic for $q\in(0,\infty)\setminus\{1\}$.   In other words, the quantum spaces $\cX_q$  $(0<q<1)$ are isomorphic as noncommutative topological spaces but not as noncommutative algebraic varieties.

\subsection*{Acknowledgements}

It is a pleasure to thank Ulrich Krähmer for discussions during a common visit to the University of Wollongong.

\section{Preliminaries}
This article is not completely self-contained.  We will make heavy use of the notation and results of \cite{NesTus:functions}, especially Sections 1, 2 and the beginning of Section 3.  

The isomorphism $\CCXq \cong \CC[\cX_{q^{-1}}]$ is well-known, \emph{cf.} \cite[Lemma 2.4.2]{NesTus:book}.  Also,  $\CCXq$ is commutative if and only if $q=1$.  We therefore work throughout with $0<p,q<1$.

%----------------------------------------
\section{$\SU_q(2)$ and the quantum disk}

The algebra $\CC[\SU_q(2)]$ is the universal $*$-algebra generated by two elements $\alpha$ and $\gamma$ satisfying
\begin{equation}
\label{eq:SUq2}
\begin{array}{c}
  \alpha\gamma = q\gamma\alpha, \quad
  \alpha\gamma^* = q\gamma^*\alpha, \quad
  \gamma\gamma^* = \gamma^*\gamma \\
  \alpha^*\alpha + \gamma^*\gamma = 1 = \alpha\alpha^* + q^2 \gamma^*\gamma.
\end{array}
\end{equation}
The $*$-subalgebra $\CC[\gamma,\gamma^*]$ is a polynomial algebra in $2$ commuting conjugate variables.  We use the notation
\[
  \alpha^{(i)} = 
  \begin{cases}
    \alpha^i, & \text{if }i\geq0,\\
    (\alpha^{*})^{-i}, & \text{if } i<0.
  \end{cases}
\]
Putting
\(
 \cA_i = \alpha^{(i)} \CC[\gamma,\gamma^*] = \CC[\gamma,\gamma^*]\alpha^{(i)}.
\) 
we get an algebra grading
\begin{equation}
\label{eq:grading}
  \CC[\SU_q(2)] = \bigoplus_{i\in\ZZ} \cA_i.
\end{equation}

The standard irreducible representation $\rho_q$ of $\CC[\SU_q(2)]$ on $\cB(\ell^2(\ZZ_+))$ is given in the basis $(e_n)_{n\in\ZZ_+}$ by 
\begin{equation}
\label{eq:Soibelman}
  \rho_q(\alpha) : e_n \mapsto \sqrt{1-q^{2n}}\,e_{n-1} , \quad
  \rho_q(\gamma) : e_n \mapsto -q^n e_n.
\end{equation}

\begin{lemma}
The kernel of $\rho_q$ is the ideal generated by $\gamma-\gamma^*$.
\end{lemma}

\begin{proof}
%Certainly $\gamma-\gamma^* \in \ker(\rho_q)$.  
Let $a\in\ker(\rho_q)$.  Write $a=\sum_i c_i \alpha^{(i)} \in \bigoplus_i \cA_i$ where $c_i \in \CC[\gamma,\gamma^*]$.  Since $\rho_q$ maps $\cA_i$ to weighted shifts of degree $i$, we have   $\rho_q(c_i\alpha^{(i)})=0$ for each $i$.  

Let $i\geq 0$.  We have $\rho_q(c_i \alpha^{(i)}\alpha^{(i)*}) = 0$.  Since $\rho_q(\alpha^{(i)}\alpha^{(i)*})$ is diagonal with no zero entries, we get $\rho_q(c_i) = 0$.  The spectrum of $\rho_q(\gamma)= \rho_q(\gamma^*)$ is $-q^{\ZZ_+}$ which is Zariski dense in $\RR$.  So if we view $c_i $ as a polynomial in $\CC[\gamma,\gamma^*] \cong \CC[Z,\overline{Z}]$, it is zero on the real line and hence divisible by $\gamma-\gamma^*$.

For $i<0$, we can take the adjoint to reduce to the previous case.  
\end{proof}

We define $\CCDq = \CC[\SU_q(2)]/\langle \gamma-\gamma^* \rangle \cong \rho_q(\CC[\SU_q(2)])$ and let $\leaf_q : \CC[\SU_q(2)] \to \CCDq$ denote the quotient map.  Then $\CCDq$ has generators
\begin{equation}
\label{eq:leaf}
  y=\leaf_q(\gamma), \quad z = \leaf_q(\alpha)^*
\end{equation}
and relations
\begin{equation}
\label{eq:CCDq}
  yz = qzy, \quad y=y^*, \quad zz^*+y^2 = 1 = z^*z + q^2 y^2.
\end{equation}
In Poisson language, $\leaf_q$ is the restriction of polynomials on $\SU_q(2)$ to a certain quantized symplectic leaf.  

\begin{remark}
The $*$-algebra $\CCDq$ embeds into the $C^*$-algebra of the quantum disk $C(\overline\DD_q)$ of \cite[\S{}3]{NesTus:functions} via the homomorphism
\begin{equation}
\label{eq:disk}
  z \mapsto Z_q, \quad  y \mapsto -(1-Z_qZ_q^*)^{\frac12}.
\end{equation}
This is not an algebraic map, which is why we add the generator $y$.  
\end{remark}

%-----------------------------
\section{Homogeneous spaces for $\SU_q(2)$}

Write $T$ for the diagonal torus in $\SU(2)$ .  For $n\in\{1,2,\ldots,\}$, let $T_n$ denote the closed subgroup of $T$ generated by the $n$th roots of unity.  We also use the convention $T_\infty = T$.  
The associated quantum homogeneous spaces  \cite[\S{}2]{NesTus:functions} are given by the quantized function algebras
\begin{equation}
\label{eq:CCDqTn}
  \CC[\SU_q(2)/T_n] = \Vect\{ \alpha^{(i)} \gamma^j \gamma^{*k} \mid i+j-k \in n\ZZ \}  \subseteq \CC[\SU_q(2)] ,
\end{equation}
where we use the convention $\infty\ZZ = \{0\}$.    Explicitly $\SU_q(2)/T_n$ is $\SU_q(2)$ when $n=1$, $\SO_q(3)$ when $n=2$ \cite{Podles:symmetries}, the quantum lens space $L^3_q(n,1)$ when $3 \leq n < \infty$ \cite{HonSzy:lens} and the standard Podle\'{s} sphere when $n=\infty$ \cite{Podles:spheres}.

Let us also write 
\[
\CCDqTn = \leaf_q(\CC[\SU_q(2)/T_n]) \cong \rho_q(\CC[\SU_q(2)/T_n]).  
\]
This is the quantization of a $2$-dimensional symplectic leaf of $\SU_q(2)/T_n$.

%----------------------------------
\section{Commutator spectrum}

\begin{definition}
\label{def:CommSp}
Define the \emph{commutator spectrum} of an algebra $\cA$ to be 
\[
 \CommSp(\cA) = \{ \omega \in \CC^\times \mid \exists a,b\in\cA\setminus\{0\} \text{ with } ab = \omega ba \}.
\]
\end{definition}

\begin{lemma}
\label{lem:Q}
  The commutator spectrum of $\CCDqTn$ is 
  \[
  \cQ(\CCDqTn) = 
  \begin{cases} 
   q^\ZZ & \text{if $n$ is odd} \\
   q^{2\ZZ} & \text{if $n$ is even or $\infty$}.
  \end{cases}
  \]
\end{lemma}

\begin{proof}
Consider first the quantum plane algebra  $\CC_q[y,z]$ generated by $y,z$ with $yz=qzy$.  We do not impose a $*$-structure for the moment.  We have
\begin{equation}
\label{eq:power}
  (y^jz^k)(y^{j'}z^{k'}) = q^{jk'-j'k} (y^{j'}z^{k'})(y^{j}z^{k}).
\end{equation}
This shows that $q^\ZZ \subseteq \CommSp[\CC_q[y,z])$.  Moreover, by considering the leading terms for the lexicographical ordering on powers of $y$ and $z$, one can deduce that $\CommSp(\CC_q[y,z]) = q^\ZZ$.  

Now suppose $\omega\in\CommSp(\CCDqTn)$, so we have $a,b\in\CCDqTn$ nonzero with $ab=\omega ba$.  By considering the highest order terms in the grading \eqref{eq:grading}, we can assume without loss of generality that $a\in\leaf_q(\cA_i)$, $b\in\leaf_q(\cA_{i'})$ for some $i,i'$.  If both $i,i'\leq 0$  then $a,b\in\CC_q[y,z]$ so $\omega\in q^\ZZ$.  If both $i,i'>0$ we can reduce to the previous case by taking adjoints.  Finally, if $i<0$ and $i'>0$ then we can reduce to the previous case by replacing $b$ by $a^mb$ for some sufficiently large $m$.

This proves that $\CommSp(\CCDqTn) \subseteq q^\ZZ$, and also that it is realized on monomials $a=y^jz^k$, $b=y^{j'}z^{k'} \in \CCDqTn$.  Note also that it is closed under taking integer powers, by considering $a$ and $b^m$.

If $n$ is even or $\infty$, then Equation \eqref{eq:CCDqTn} forces $j\equiv k ~[2]$ and $j' \equiv k'~[2]$.  Therefore, by Equation \eqref{eq:power}, $\CommSp(\CCDqTn) \subseteq q^{2\ZZ}$.  Putting $a=\leaf_q(\gamma^*\alpha)  = yz^*$, $b=\leaf_q(\gamma^*\gamma) = y^2$ gives $\omega = q^2$.

If $n$ is odd, we can put $a=\leaf_q(\gamma^{n+1}\alpha^*) = y^{n+1}z$,  $b=\leaf_q(\gamma^{\frac{n+1}{2}} \gamma^{*\frac{n-1}{2}}\alpha^*) = y^{n}z$ and we get $\omega = q$.   
\end{proof}

%-----------------------------
\section{Quantized homogeneous spaces}

Let $\Pi = \{\alpha_1,\ldots,\alpha_r\}$ be the simple roots of $G$ and $\{\omega_1,\ldots,\omega_r\}$  the associated fundamental weights.  We write $(\slot,\slot)$ for the $G$-invariant bilinear form on $\lie{h}^*$ with $(\alpha,\alpha)=2$ for all short roots $\alpha$, and we use the standard notation $q_i = q^{(\alpha_i,\alpha_i)/2}$.  Thus $q_i = q$, $q^2$ or $q^3$ depending on the length of $\alpha_i$.  
For every $\alpha_i\in\Pi$, there is a restriction map $\sigma_i : \CCGq \to \CC[\SU_{q_i}(2)]$.

The Poisson-Lie subgroups of $G$ are determined by pairs $(S,L)$ where $S$ is a subset of the simple roots for $G$ and $L$ is a subgroup of the lattice $P(S^c) = \ZZ \{ \omega_i \mid \alpha_i \in S^c \}$.  
We refer to \cite[Proposition 1.1]{NesTus:functions} for the details.
Fix such a subgroup $K^{S,L}$ and consider $\cX = G/K^{S,L}$ and its quantized algebra of function $\CCXq$.

\begin{lemma}
\label{lem:restrictions}
For every $\alpha_i\in S^c$ we have $\sigma_i(\CCXq) = \CC[\SU_{q_i}(2)/T_{n_i}]$ for some $n_i \in \{1,2,\ldots,\infty\}$.
\end{lemma}

\begin{proof}
This follows from \cite[Corollary 2.4]{NesTus:functions}.
Explicitly, $n_i$ is the generator of the subgroup of $\ZZ$ obtained by restricting the weights $\mu\in L$ to weights for $\cU_{q_i}(\lie{sl}_2)$, \emph{i.e.},
\(
  n_i\ZZ =   \{ (\mu,\alpha_i) \mid \mu\in L \} .
\)
\end{proof}

Let $T$ be the maximal torus of $G$, and $T_L \subseteq T$ the subgroup annihilated by $L\subseteq P = \widehat{T}$.
Let $W^S \subset W$ be the set of $w\in W$ such that $w(\alpha_i) >0$ for all $\alpha_i\in S$.

The irreducible representations of $\CCXq$ are indexed by pairs $(w,t) \in W^S\times T/T_L$.  These representations $\pi_{w,t}$
%\[
%  \pi_{w,t} : \CC[\cX_q] \to \cB\left( \ell^2({\ZZ_+})^{\otimes \ell(w)}\right),
%\]
%where $\ell(w)$ is the Bruhat length of $w\in W$, 
are described explicitly in the preamble to Theorem 2.2 of \cite{NesTus:functions}.  See also \cite[Theorem 5.9]{StoDij}.  We recall only that in the case where $w=s_i$ is the simple reflection in $\alpha_i\in S^c$, the associated representation is 
\begin{equation}
\label{eq:pi_s_i}
  \pi_{s_i,1} = \rho_q\circ\sigma_i : \CCXq \to \CC[\DD_{q_i}/T_{n_i}] \to \cB(\ell^2(\ZZ_+)).
\end{equation}

%-----------------------------
\section{Quantum $2$-cell representations}

\begin{definition}
An irreducible representation $\pi$ of $\CCXq$ will be called a \emph{$2$-cell representation} if the norm closure of the image is an extension of an abelian $C^*$-algebra by the compacts---that is, if it fits into an exact sequence of $C^*$-algebras
\[
 0 \to \cK \to \pi(C(\cX_q)) \to A \to 0
\]
with $A$ abelian and $\cK$ being the compacts on a separable infinite dimensional Hilbert space.
\end{definition}

\begin{example}
\label{ex:2-cell}
The standard representation $\rho_q$ of $\CC[\SU_q(2)/T_n]$  is a $2$-cell representation  for any $n\in\{1,2,\ldots,\infty\}$.  To see this, let $I$ denote the $*$-ideal of $\CC[\SU_q(2)]$ generated by $\gamma$ and put $I_n = I \cap \CC[\SU_q(2)/T_n]$.  Then $\overline{\rho_q(I_n)} = \cK(\ell^2({\ZZ_+}))$ since it contains $\rho_q(\alpha^*\gamma)$ which is a compact weighted shift with distinct nonzero weights.  
And putting $\gamma=0$ in the relations \eqref{eq:SUq2} shows that $\CC[SU_q(2)/T_n] / I_n$ is abelian.
\end{example}

\begin{lemma}
\label{lem:2-cell}
  The representation $\pi_{w,t}$ is a $2$-cell representation if and only if $w$ is a simple reflection, \emph{i.e.} $w=s_i$ for some $\alpha_i \in S^c$.
  \end{lemma}

\begin{proof}
  The representation $\pi_{w,t}$ is the representation $\pi_{w,1}$ twisted by a character, so the image of $\pi_{t,w}(\CCXq)$ is independent of $t$.  Therefore, we may take $t=1$.  
  
  If $w=1$, then $\pi_{1,1}$ is the counit.
  
  If $w=s_i$ for $i\in S^c$, then by Lemma \ref{lem:restrictions} $\pi_{w,1}(\CCXq)   = \rho_{q_i}(\CC[\SU_{q_i}(2)/T_{n_i}])$, so $\pi_w$ is a $2$-cell representation by Example \ref{ex:2-cell}.
  
  If $w\in W^S$ has Bruhat length two or more, then \cite[Theorem 4.1]{NesTus:functions} shows that $\pi_{w,1}(C(\cX_q))$ has a nontrivial nonabelian quotient, so $\pi_{w,1}$ is not a $2$-cell representation.

\end{proof}

%----------------------------
\section{Proof of the main theorem}

\begin{proof}[Proof of \Cref{theorem}]

By Lemmas  \ref{lem:Q}, \ref{lem:restrictions} and \ref{lem:2-cell}, the greatest value, less than one, of the commutator spectrum of the image of any $2$-cell representation of $\CCXq$ is $q^m$, for some $m\in\NN$ independent of $q$.  This is an invariant of $\CCXq$.
\end{proof}

% -------------------------------------------------------

\printbibliography

\end{document}